\documentclass[11pt]{amsart}
\usepackage{graphicx}
\usepackage{pstricks}
\usepackage{pst-plot}
\usepackage{pgf}
\usepackage{pst-node}
\usepackage{tikz}

\newcommand{\nc}{\newcommand}

\nc{\vg}{\mathfrak{v} } \nc{\wg}{\mathfrak{w} } \nc{\zg}{\mathfrak{z} }
\nc{\ngo}{\mathfrak{n} } \nc{\ngoq}{\mathfrak{n}^{\QQ} }
\nc{\ngoz}{\mathfrak{n}^{\ZZ} } \nc{\ggoq}{\mathfrak{g}^{\QQ} }
\nc{\kg}{\mathfrak{k} } \nc{\mg}{\mathfrak{m} } \nc{\bg}{\mathfrak{b} }
\nc{\ggo}{\mathfrak{g} } \nc{\ggob}{\overline{\mathfrak{g}} }
\nc{\sog}{\mathfrak{so} } \nc{\sug}{\mathfrak{su} } \nc{\spg}{\mathfrak{sp}}
\nc{\slg}{\mathfrak{sl} } \nc{\glg}{\mathfrak{gl} } \nc{\cg}{\mathfrak{c} }
\nc{\rg}{\mathfrak{r} } \nc{\hg}{\mathfrak{h} } \nc{\tg}{\mathfrak{t} }
\nc{\ug}{\mathfrak{u} } \nc{\dg}{\mathfrak{d} } \nc{\ag}{\mathfrak{a} }
\nc{\pg}{\mathfrak{p} } \nc{\sg}{\mathfrak{s} } \nc{\lgo}{\mathfrak{l} }
\nc{\fg}{\mathfrak{f} }

\nc{\pca}{\mathcal{P}} \nc{\nca}{\mathcal{N}} \nc{\lca}{\mathcal{L}}
\nc{\oca}{\mathcal{O}} \nc{\mca}{\mathcal{M}} \nc{\tca}{\mathcal{T}}
\nc{\aca}{\mathcal{A}}
\nc{\GRA}{\mathcal{G}}

\nc{\vp}{\varphi} \nc{\ddt}{\frac{{\rm d}}{{\rm d}t}} \nc{\im}{\mathtt{i}}

\nc{\ala}{Anosov Lie algebra} \nc{\alas}{Anosov Lie algebras}

\nc{\SO}{\mathrm{SO}} \nc{\Spe}{\mathrm{Sp}} \nc{\Sl}{\mathrm{SL}}
\nc{\SU}{\mathrm{SU}} \nc{\Or}{\mathrm{O}} \nc{\U}{\mathrm{U}} \nc{\Gl}{\mathrm{GL}}
\nc{\Se}{\mathrm{S}} \nc{\Cl}{\mathrm{Cl}} \nc{\Spein}{\mathrm{Spin}}
\nc{\Pin}{\mathrm{Pin}} \nc{\G}{\mathrm{GL}_n} \nc{\g}{\mathfrak{gl}_n}

\nc{\RR}{{\Bbb R}} \nc{\HH}{{\Bbb H}} \nc{\CC}{{\Bbb C}} \nc{\ZZ}{{\Bbb Z}}
\nc{\FF}{{\Bbb F}} \nc{\NN}{{\Bbb N}} \nc{\QQ}{{\Bbb Q}} \nc{\PP}{{\Bbb P}}

\nc{\vs}{\vspace{.5cm}}
 \nc{\ip}{\langle\cdot,\cdot\rangle}
\nc{\la}{\langle} \nc{\ra}{\rangle} \nc{\unm}{\frac{1}{2}} \nc{\unc}{\frac{1}{4}}
\nc{\und}{\frac{1}{16}} \nc{\f}{\frac}

\nc{\no}{\vs\noindent} \nc{\lam}{\Lambda^2\ggo^*\otimes\ggo} \nc{\tang}{{\rm T}}
\nc{\dif}{{\rm d}} \nc{\preq}{\simeq_K} \nc{\lb}{[\cdot,\cdot]}
\nc{\ngog}{\ngo_{\mathcal{G}}}

\nc{\He}{\operatorname{Hess}} \nc{\ad}{\operatorname{ad}}
\nc{\Ad}{\operatorname{Ad}} \nc{\rank}{\operatorname{rank}}
\nc{\Irr}{\operatorname{Irr}} \nc{\End}{\operatorname{End}}
\nc{\Aut}{\operatorname{Aut}} \nc{\Inn}{\operatorname{Inn}}
\nc{\Der}{\operatorname{Der}} \nc{\Ker}{\operatorname{Ker}}
\nc{\Iso}{\operatorname{I}} \nc{\Diff}{\operatorname{Diff}}
\nc{\Lie}{\operatorname{L}} \nc{\tr}{\operatorname{tr}}

\nc{\degr}{\operatorname{dgr}} \nc{\sen}{\operatorname{sen}}
\nc{\modu}{\operatorname{mod}} \nc{\Ric}{\operatorname{Ric}}
\nc{\sym}{\operatorname{sym}} \nc{\sca}{\operatorname{sc}} \nc{\scalar}{{\sf s}}
\nc{\grad}{\operatorname{grad}} \nc{\ricci}{\operatorname{ric}}
\nc{\Rin}{\operatorname{M}} \nc{\Le}{\operatorname{L}}
\nc{\level}{\operatorname{level}} \nc{\rad}{\operatorname{r}}
\nc{\abel}{\operatorname{ab}} \nc{\Pf}{\operatorname{Pf}}

\theoremstyle{plain}
\newtheorem{theorem}{Theorem}[section]
\newtheorem{proposition}[theorem]{Proposition}

\newtheorem{lemma}[theorem]{Lemma}

\theoremstyle{definition}

\theoremstyle{remark}
\newtheorem{remark}[theorem]{Remark}

\begin{document}

\title[On the existence of nilsolitons]{On the existence of nilsolitons on $2$-step nilpotent Lie groups}%
\author{David Oscari}%
\address{}%
\email{}%

\thanks{}%

\begin{abstract}  A 2-step nilpotent Lie algebra $\ngo$ is said to be of {\it type} $(p, q)$ if $\dim \ngo=p+q$ and $\dim [\ngo,\ngo]=p$.  By considering a class of $2$-step nilpotent Lie algebras naturally attached to graphs, we prove that there exist indecomposable, 2-step nilpotent Lie groups of type $(p,q)$ which do not admit a nilsoliton metric for every pair $(p,q)$ such that $21\leq q$ and $q-1\leq p\leq \frac{1}{2}q^2-\frac{5}{2}q+9$.  This improves a result due to Jablonski \cite{Jbl}.
\end{abstract}
\maketitle

\section{Introduction}

A nilpotent Lie algebra $\ngo$ is said to be an {\it Einstein nilradical} if it admits an inner product $\ip$ such that
$\text{Ric}_{\ip}=cI+D$ for some $c\in\RR$ and $D\in\Der(\ngo)$, where $\text{Ric}_{\ip}$ is the Ricci operator of the left-invariant Riemannian metric defined by $\ip$ on the simply connected nilpotent Lie group $N$ with Lie algebra $\ngo$.  Such metrics are called {\it nilsolitons} in the literature and play the role of most distinguished or canonical metrics on nilmanifolds, as they are proved in \cite{Lrt} to satisfy
the following properties:

\begin{itemize}
\item They are {\it Ricci solitons}, i.e. the solutions of the
Ricci flow starting at them evolve only by scaling and
the action by diffeomorphisms (see \cite[Chapter 1]{libro}).

\item A given $N$ can admit at most one nilsoliton up to isometry and scaling among all its left-invariant metrics.

\item Einstein nilradicals are precisely the nilpotent parts of Einstein solvmanifolds.
\end{itemize}

Nevertheless, the existence, structural and classification problems on nilsolitons
seem to be far from being satisfactorily solved, if solved at all (see the survey \cite{cruzchica} for further information). It is proved, for instance, in \cite{Pyn} that in any dimension $\geq 8$ there is a one-parameter family of pairwise non-isomorphic $\mathbb{N}$-graded nilpotent Lie algebras which are not Einstein nilradicals.

In this paper, we are concerned with the following question: how are the Einstein and non-Einstein nilradicals distributed among 2-step nilpotent Lie algebras?
We are mainly interested in algebras which are indecomposable, in the sense that
they can not be written as a direct sum of ideals, as it is known that $\ngo=\ngo_1\oplus\ngo_2$ is an Einstein nilradical if and only if both $\ngo_1$ and $\ngo_2$ are so (see \cite{Jbl2}, \cite{Nkl2}).

A 2-step nilpotent Lie algebra $\ngo$ is said to be of {\it type} $(p, q)$ if $\dim \ngo=p+q$ and $\dim [\ngo,\ngo]=p$.  It follows that always
$p\leq D_q:=\frac{1}{2}q(q-1)$.  There is only one 2-step nilpotent Lie algebra of type $(D_q,q)$ and only finitely many of type
$(D_q-1,q)$ (up to isomorphism), which are all Einstein nilradicals (see \cite{Nkl2}).  On the other hand, any $2$-step nilpotent Lie algebra of dimension $\leq 7$ is an Einstein nilradical (see \cite{Wll} and \cite{Frn}).  Recently, by using concatenation and adjoint techniques of Lie algebras, Jablonski proved the following result.

\begin{theorem}\cite{Jbl}\label{Ja}
There exist indecomposable, 2-step nilpotent Lie algebras of type $(p,q)$, which are not Einstein nilradicals, for every pair $(p,q)$ such that
$$
8\leq q, \qquad\mbox{and}\qquad 2\leq p\leq \frac{5}{4}q-8.
$$
\end{theorem}

We study in this note a class of $2$-step nilpotent Lie algebras naturally attached to graphs in order to find non-Einstein nilradicals among the types which are not covered by the above theorem.  We use a criterium based on the `positivity' of a graph given in \cite{LrtWll}.  Our main result can be stated as follows.

\begin{theorem}\label{thm}
There exist indecomposable, 2-step nilpotent Lie algebras of type $(p,q)$, which are not Einstein nilradicals, for every pair $(p,q)$ such that
$$
21\leq q \qquad\mbox{and}\qquad q-1\leq p\leq \frac{1}{2}q^2-\frac{5}{2}q+9.
$$
\end{theorem}

As it can be visualized in Figure \ref{Figura}, this boundary considerably improves the one given in Theorem \ref{Ja} for large $q$, as it is quadratic instead of linear.  The existence of non-Einstein nilradicals of type
$(p,q)$ such that
$$
\frac{1}{2}q^2-\frac{5}{2}q+10\leq p\leq D_q-2.
$$
is an open question.

\begin{figure}
\begin{center}
\begin{tikzpicture}[rotate=0,scale=0.8,domain=2:700,xscale=0.01mm,yscale=.05mm]
\draw[-] (0,-1) -- (0,30) node[left] {$q$};
\draw[-] (-3,0) -- (400,0) node[below] {$p$};
\draw[color=red,domain=5.6:25] plot (0.5*\x*\x-0.5*\x,\x) node[right] {$p = D_q = \frac{1}{2}q^2-\frac{1}{2}q$};
\draw[color=blue,domain=6:30] plot (0.5*\x*\x-2.5*\x+9,\x) node[right] {$p = \frac{1}{2}q^2-\frac{5}{2}q+9$};
\draw[color=green,domain=8:30] plot (1.25*\x-8,\x) node[right] {$p = \frac{5}{4}q-8$};
\end{tikzpicture}
\end{center}  \caption{}\label{Figura}
\end{figure}

{\it Acknowledgements.} I am grateful to my advisor Jorge Lauret for his invaluable help during the preparation of the paper.


\section{ Preliminaries}

\subsection{ The Lie algebra associated with a graph}

Let $\mathcal{G}=(S,E)$ be a (finite, undirected) graph, with set of vertices $S=\{v_1,\dots,v_q\}$
and edges $E=\{l_1,\dots,l_p\}$, $l_k=v_iv_j$ for some $i,j$. We associate the Lie algebra
$\ngog=(\RR^n,\lb)$, $ n=p+q$ with each $\mathcal{G}$, with Lie bracket defined by
$$
[e_i,e_j]=\left\{\begin{array}{cl}
                   e_{q+k} & \textrm{if }l_k=v_iv_j, i<j\leq q;\\
                   0 & \textrm{otherwise}.
                 \end{array}
\right.
$$
where $\{e_i\}_{i=1}^n$ is the standard basis of $\RR^n$. We will often identify the vertices of the
graph with the vectors $e_1,\dots,e_q$,  and the edges with $e_{q+1},\dots,e_{q+p}$. Then the bracket between two vertices $v_i$ and $v_j$, $i<j$, is the edge joining them, and it vanishes otherwise. To obtain a well defined bracket we add the assumption that no two edges join the same pair of vertices.

Recall from graph theory that two edges $l_k$, $l_m$ of a graph $\GRA$ are called adjacent if
they share a vertex, which will be denoted by $l_k\sim l_m$. The line graph $L(\GRA)$ of $\GRA$ is
the graph whose vertices are the edges of $\GRA$  and where two of them are joined if and
only if they are adjacent. The adjacency matrix $\textrm{Adj}\,(\mathcal{G})$ of a graph $\GRA$ with a labelling
$\{v_1,\dots,v_q\}$ for its set of vertices is defined as the (symmetric) $q\times q$ matrix with 1
in the entry $i,j$ if $v_iv_j$ is an edge and zero otherwise.

\begin{proposition}\cite{LrtWll}\label{criteriopositividad} $\ngo_{\mathcal{G}}$ is an Einstein nilradical if and only if there exist $\nu>0$ and weights $c_1,\dots, c_p \in\RR$ for the edges such that
 \begin{align}
\label{equivalenciasistema}3c_k+\sum_{l_m\sim l_k}c_m=\nu, & \qquad \forall k=1,\dots,p,
\\\label{entradaspositivas}  c_k>0, & \qquad \forall k=1,\dots,p,
\end{align}
where the sum is over all edges $l_m$ that share a vertex with $l_k$.
\end{proposition}
A graph satisfying properties (1) and (2) is called {\it positive}. If we consider the line graph
$L(\mathcal{G})$ of $\GRA$, the first condition above may be written in terms of its adjacency matrix
$\textrm{Adj}\, L(\mathcal{G})$ as
\begin{equation*}\label{positividad}
\left(3I +\textrm{Adj}\, L(\mathcal{G})\right)\left[
                                                \begin{array}{c}
                                                  c_1 \\
                                                  \vdots \\
                                                  c_p \\
                                                \end{array}
                                              \right]=\nu\left[
                                                           \begin{array}{c}
                                                             1 \\
                                                             \vdots \\
                                                             1 \\
                                                           \end{array}
                                                         \right].
                    \end{equation*}

It can be proved that the matrix $3I+\textrm{Adj}\, L(\mathcal{G})$ is positive definite, thus given $\nu>0$
the above system has always a unique solution. And since $\nu>0$, we have that $\GRA$ is positive if and only if
$$
\left(3I +\textrm{Adj}\, L(\mathcal{G})\right)^{-1}\left[
                                                           \begin{array}{c}
                                                             1 \\
                                                             \vdots \\
                                                             1 \\
                                                           \end{array}
                                                         \right]
$$
has all its entries positive.

\subsection{Coherent decomposition of a graph}
Let $\mathcal{G}=(S,E)$ be a graph, and let us define for each $\alpha\in S$,
$$
\Omega'(\alpha)=\{\omega\in S:\, \omega\alpha\in E\}\quad\textrm{and}\quad \Omega(\alpha)=\Omega'(\alpha)\cup \{\alpha\}.
$$
Now consider the equivalence relation $\sim$ in $S$ defined as follows:
$$
\alpha\sim\beta\quad\textrm{if and only if}\quad \Omega'(\alpha)\subseteq\Omega(\beta)\textrm{ and }
\Omega'(\beta)\subseteq\Omega(\alpha),
$$
or in other words, two vertices are related if and only if they have the same neighbors. Let $\Lambda=\Lambda(S,E)$ be the set of equivalence classes in $S$ with respect to $\sim$, for each $\lambda\in\Lambda$ we call $S_{\lambda}\subseteq S$ its equivalence class. The subsets $S_{\lambda}$, $\lambda\in\Lambda$, are the coherent
components of $(S,E)$; they form a partition of the set $S$.

This decomposition was considered in \cite{DnMnk}, where the following properties are
also mentioned:
\begin{itemize}
  \item  Given $\mathcal{G}=(S,E)$, with $S_{\lambda}$ its coherent components, it is easy to see
that if for a given $\lambda\in\Lambda$ there exist $\alpha,\beta\in S_{\lambda}$ such that $\alpha\beta\in E$, then $\xi\eta\in E$ for all $\xi,\eta\in S_{\lambda}$. This implies that a coherent component is on its own either a complete graph or a discrete one.
  \item  To generalize the previous item let us assume that, given $\lambda,\mu\in\Lambda$  there exist $\alpha\in S_{\lambda}$ and $\beta\in S_{\mu}$, such that $\alpha\beta\in E$.  Then it is easy to see that $\xi\eta\in E$ for all $\xi\in S_{\lambda}$, $\eta\in S_{\mu}$.  Therefore, given two coherent components $S_{\lambda}$ and $S_{\mu}$, either they are not adjacent at all, or every possible edge between them is present in $E$. In the latter case we say that $S_{\lambda}$ and $S_{\mu}$ are {\it adjacent}.
      Let us define a set of unordered pairs $\mathcal{E}$ in such a way that $\lambda\mu\in\mathcal{E}$ if
and only if the components $S_{\lambda}$ and $S_{\mu}$ are adjacent. We call $(\Lambda,\mathcal{E})$  the {\it coherence graph} associated with $(S,E)$.
\end{itemize}
These properties give us the following useful result on the weights of a general
graph. We call two edges {\it similar} if either they join the same pair of coherent components, or they are in the same coherent component.

 \begin{proposition}\cite[Proposition 2.10]{Lfn}\label{ladossimilares}  Let $\mathcal{G}=(S,E)$ be a positive graph, with weights $(c_i)_{i=1}^p$ for some $\nu>0$ fixed. If $l_i,l_j$ are two similar edges, then $c_i=c_j$.
 \end{proposition}

\begin{remark}\label{solocompnotriviales}
In an arbitrary graph, if $l_i,l_j$ are edges with weights $c_i,c_j$ respectively, then the corresponding equations to $l_i,l_j$ coincide in the system (\ref{equivalenciasistema}) by Proposition \ref{ladossimilares}. Therefore, for a given graph $(S,E)$, the system (\ref{equivalenciasistema}) can be rewritten, obtaining a system in $|\mathcal{E}|+|\{\lambda\in\Lambda_0:\, |S_{\lambda}|>1\}|$ variables.
\end{remark}

\subsection{Results on two and three coherent components}\label{resultados} In \cite[Table 1]{Lfn}, the classification of graphs with up to 3 coherent components according to positivity is given. We revisit here this classification for self-containness, and also because we need to add a few cases where one coherent component has only one vertex.

We represent a graph via its coherence graph. Each circle represents a coherent
component, being black if the correspondent component is a complete graph, and
white if it is discrete. The existence of an edge joining two circles represents the
fact that every edge joining vertices between those coherent components is present
in the original graph. Finally, the natural number near to each coherent component
is the number of vertices that it contains.

\subsubsection{}\label{o-*} We consider the graph:

\begin{center}
\begin{pspicture}(-1.3,-.2)(1.3,.4)
\psline(-.75,0)(.75,0)
\psdots[dotstyle=o,dotsize=5pt](-.75,0)
\psdots[dotstyle=*,dotsize=5pt](.75,0)

\rput[l](-1.1,-.2){$r$}
\rput[l](1,-.2){$s$}
\end{pspicture}
\end{center}
We denote $S_1,S_2$ the coherent components with $r,s$ vertices, respectively. By
Proposition \ref{ladossimilares}, there are only two possibly different edge weights in this case: $a$,
for the edges joining $S_1$ with $S_2$, and $b$, for the edges inside $S_2$. Now, according to Remark \ref{solocompnotriviales}, we must consider two cases: $r\geq 1,s>1$, and $r\geq 1,s=1$.

If $r\geq 1,s>1$, the graph is positive if and only if $s\geq r$, due to \cite[Table 1]{Lfn}.

Now if $r\geq 1$ and $s=1$, all equations in the system (\ref{equivalenciasistema}) agree with the single equation $(2+r)a=\nu$, and its solution is $a=\nu/(2+r)$. The graph is positive if and only if $a$ is positive, which is clearly true, since $\nu>0$.

It follows that the graph is positive if and only if $s\geq r$ or $s=1$.

\subsubsection{}\label{o-*-*} Now we consider
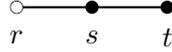
\begin{center}
\begin{pspicture}(-.5,-.3)(2.5,.5)
\psline(0,0)(2,0)
\psdots[dotstyle=o,dotsize=5pt](0,0)
\psdots[dotstyle=*,dotsize=5pt](1,0)
\psdots[dotstyle=*,dotsize=5pt](2,0)

\rput[d](0,-.4){$r$}
\rput[d](1,-.4){$s$}
\rput[d](2,-.4){$t$}

\end{pspicture}\textrm{$$$$}
\end{center}
Let us call $S_1,S_2,S_3$ the coherent components with $r,s,t$ vertices, respectively. By
Proposition \ref{ladossimilares}, there are only four possibly different edge weights in this case: $a,b$,
for the edges joining $S_1$ with $S_2$, and $S_2$ with $S_3$ respectively; and $c,d$ for the edges inside $S_2$ and $S_3$ respectively. Now, according to Remark \ref{solocompnotriviales}, we must consider two cases: $r\geq 1,s>1$, and $r\geq 1,s=1$. In total we must consider four cases:
 \begin{itemize}
   \item[(i)] $r\geq 1,s>1,t>1$;
   \item[(ii)] $r\geq 1,s>1,t=1$;
   \item[(iii)] $r\geq 1,s=1,t>1$;
   \item[(iv)] $r\geq 1,s=1,t=1$.
 \end{itemize}

In (i), if $r\geq 1,s>1,t>1$, the graph is positive if and only if $(s+t)(s-r)>(r-1)(t-1)$, due to \cite[Table 1]{Lfn}.

In cases (ii) and (iv), i.e. $t=1$ and $s\geq 1$, the graph is actually
\begin{center}
\begin{pspicture}(-1.3,-.4)(1.3,.4)
\psline(-.75,0)(.75,0)
\psdots[dotstyle=o,dotsize=5pt](-.75,0)
\psdots[dotstyle=*,dotsize=5pt](.75,0)

\rput[l](-1.7,-.2){$r+1$}
\rput[l](1,-.2){$s$}
\end{pspicture}
\end{center}
which is positive if and only if $s=1$, or $s\geq r+1$, i.e. $s-r>0$.

In (iii), rewriting the system (\ref{equivalenciasistema}), we obtain
$$
\left[
  \begin{array}{ccc}
    r+2 & t & 0 \\
    r & t+2 & t-1 \\
    0 & 2 & 2t-1 \\
  \end{array}
\right]\left[
         \begin{array}{c}
           a \\
           b \\
           d \\
         \end{array}
       \right]
       =\nu\left[
          \begin{array}{c}
            1 \\
            1 \\
            1 \\
          \end{array}
        \right]\qquad(\nu>0)
$$
Then
$$
\left[
  \begin{array}{c}
    a \\
    b \\
    d \\
  \end{array}
\right]=\frac{\nu}{2t(2t+r+1)}\left[
                                \begin{array}{c}
                                  2(r+t) \\
                                  (1+t)t \\
                                  2t+r-rt \\
                                \end{array}
                              \right].
$$
The graph is positive if and only if $a,b,d>0$, that is, if and only if $2t+r-rt>0$, i.e. $r<\frac{2t}{t-1}=2+\frac{2}{t-1}$. The right-hand side  is decreasing. If $t=2$, $r=1,2,3$; and if $t\geq 3$, $r=1,2$. Then, $2t+r-rt>0$ if and only if $(r,t)=(3,2),(1,t),(2,t)$ with $t\geq 2$.

Therefore, from cases (i)-(iv) follows that the graph is positive if and only if some of the following holds:
\begin{itemize}
  \item $(s+t)$$(s-r)>(r-1)(t-1)$;
  \item $s=t=1$;
  \item $s=1$ and $(r,t)=(3,2),(1,t),(2,t)$ with $t\geq 2$.
\end{itemize}

\section{Sufficient conditions for non-positivity}
In this section, we give sufficient conditions for certain graphs with 4 or 5 coherent components to be non-positive by using the same idea as in \cite{Lfn}: consider its coherent decomposition and apply Proposition \ref{criteriopositividad} to rewrite the system (\ref{equivalenciasistema}), obtaining in our case a smaller system with size at most $5\times5$ (see Remark \ref{solocompnotriviales}).

\subsection{Four coherent components}
\begin{center}
\begin{figure}
\begin{pspicture}(0,-2)(12,1)
\psline(0,0)(3,0)
\psdots[dotstyle=*,dotsize=5pt](0,0)
\psdots[dotstyle=o,dotsize=5pt](1,0)
\psdots[dotstyle=*,dotsize=5pt](2,0)
\psdots[dotstyle=o,dotsize=5pt](3,0)

\rput[d](0,-.4){$1$}
\rput[d](1,-.4){$s$}
\rput[d](2,-.4){$1$}
\rput[d](3,-.4){$u$}
\psline(5,0)(6,0)(7,1)(7,-1)(6,0)
\psdots[dotstyle=o,dotsize=5pt](5,0)
\psdots[dotstyle=*,dotsize=5pt](6,0)
\psdots[dotstyle=o,dotsize=5pt](7,1)
\psdots[dotstyle=*,dotsize=5pt](7,-1)
\rput[d](5,-.4){$r$}
\rput[d](6,-.4){$s$}
\rput[c](7,1.4){$t$}
\rput[d](7,-1.4){$u$}
\psline(9,0)(10,0)(11,1)(12,0)(11,-1)(10,0)(12,0)
\psdots[dotstyle=o,dotsize=5pt](9,0)
\psdots[dotstyle=*,dotsize=5pt](10,0)
\psdots[dotstyle=*,dotsize=5pt](11,1)
\psdots[dotstyle=*,dotsize=5pt](12,0)
\psdots[dotstyle=*,dotsize=5pt](11,-1)
\rput[d](9,-.4){$r$}
\rput[d](10,-.4){$2$}
\rput[d](11,1.4){$1$}
\rput[d](12.4,0){$u$}
\rput[d](11,-1.4){$v$}
\end{pspicture}
  \caption{}\label{Fig.}
\end{figure}
\end{center}

\begin{lemma}\label{4comp 1} The graph on the left in Figure \ref{Fig.} is non-positive for all $u\geq 6$, $s=1,2$.
\end{lemma}

\begin{proof} By Proposition \ref{ladossimilares}, there are only two possibly different edge weights: $a$
for the edges joining the first coherent component with the second coherent component (from left to right);
$b$ for the edges joining the second component with the third component; and $c$ for the edges joining the third component with the fourth component. Then the system (\ref{equivalenciasistema}) can be rewritten as
$$
\left[
  \begin{array}{ccc}
    2+s & 1 & 0 \\
    1 & 2+s & u \\
    0 & s & 2+u \\
  \end{array}
\right]
\left[
  \begin{array}{c}
    a \\
    b \\
    c \\
  \end{array}
\right]=\nu\left[
             \begin{array}{c}
               1 \\
               1 \\
               1 \\
             \end{array}
           \right].
$$
Then,
$$
\left[
  \begin{array}{c}
    a \\
    b \\
    c \\
  \end{array}
\right]
=\frac{\nu}{6+3u+8s+2su+2u^2}\left[
                   \begin{array}{c}
                     2(1+s+u) \\
                     2+2s-u \\
                     3(1+s) \\
                   \end{array}
                 \right].
$$
Suppose instead that the graph is positive. In particular, we must have $b>0$, which is a contradiction because $u\geq 6$. Then the graph is not positive.
\end{proof}

\begin{lemma}\label{4comp 2} If one of the following holds:
\begin{enumerate}
  \item[(i)] $r\geq 2,s=t=2$;
  \item[(ii)] $t\geq 2,s=u=1$,
\end{enumerate}
then the graph in the middle in Figure \ref{Fig.} is not positive.
\end{lemma}
\begin{proof} Let us call $S_1,S_2,S_3,S_4$ the coherent components with $r,s,t,u$ vertices. By Proposition \ref{ladossimilares}, there are only six possibly different edge weights in this case: $a,b,c,d$, for the edges joining $S_1$ with $S_2$, $S_2$ with $S_3$, $S_3$ with $S_4$, and $S_2$ with $S_4$ respectively; and $e,f$ for the edges inside $S_2$ and $S_4$ respectively.

In (i) we must consider the subcases $u>1$ and $u=1$.

Let $r\geq 2,s=t=2$ with $u>1$. Rewriting the system (\ref{equivalenciasistema}) we obtain
$$
\left[
  \begin{smallmatrix}
    2r & 4 & 0 & 2u & 3 & 0 \\
    0 & 0 & 4 & 4 & 0 & 2u-1 \\
    3+r & 2 & 0 & u & 1 & 0 \\
    r & 5 & u & u & 1 & 0 \\
    0 & 2 & u+3 & 2 & 0 & u-1 \\
    r & 2 & 2 & u+3 & 1 & u-1 \\
  \end{smallmatrix}
\right]
\left[
         \begin{smallmatrix}
           a \\
           b \\
           c \\
           d \\
           e \\
           f\\
         \end{smallmatrix}
       \right]
       =\nu\left[
             \begin{smallmatrix}
               1 \\
               1 \\
               1 \\
               1 \\
               1 \\
               1 \\
             \end{smallmatrix}
           \right].
$$
If we call $A$ the matrix of the system, its solution is given by
$$
\left[
  \begin{smallmatrix}
    a \\
    b \\
    c \\
    d \\
    e \\
    f \\
  \end{smallmatrix}
\right]
=\frac{\nu}{\det(A)}\left[
                       \begin{smallmatrix}
                          18+u^3+21u+8u^2\\
                          -18-15u+2ru-3u^2+ru^2\\
                          3(6+2r+5u+ru+u^2)\\
                          ru^2+2ru+3r-3u^2-12u-9\\
                         (2r-3)u^2+(7r-12)u+9(r-1) \\
                           3(3+r+4u+ru+u^2)\\
                       \end{smallmatrix}
                     \right],
$$
where $\det(A)<0$ for all $r,u$ (note that $A$ is not the matrix of the system (\ref{equivalenciasistema})). If we assume that the graph is positive, we must have $e>0$, which is a contradiction because if $r\geq 2$, then $\frac{\nu}{\det(A)}[(2r-3)u^2+(7r-12)u+9(r-1)]<0$.

If $r\geq 2,s=t=2$ with $u=1$, then
$$
\left[
  \begin{smallmatrix}
    0 & 2 & 4 & 2 & 0 \\
    2r & 4 & 0 & 2 & 3 \\
    3+r & 2 & 0 & 1 & 1 \\
    r & 5 & 1 & 1 & 1 \\
    r & 2 & 2 & 4 & 1 \\
  \end{smallmatrix}
\right]\left[
         \begin{smallmatrix}
           a \\
           b \\
           c \\
           d \\
           e \\
         \end{smallmatrix}
       \right]
=
\nu
\left[
  \begin{smallmatrix}
    1 \\
    1 \\
    1 \\
    1 \\
    1 \\
  \end{smallmatrix}
\right],
$$
where
$$
\left[
  \begin{smallmatrix}
    a \\
    b \\
    c \\
    d \\
    e\\
  \end{smallmatrix}
\right]
=\frac{\nu}{264+18r}\left[
                      \begin{smallmatrix}
                        48\\
                          -3(r-12)\\
                        9(r+4) \\
                        -6(r-4) \\
                         -6(3r-4)\\
                      \end{smallmatrix}
                    \right].
$$
As before, if the graph is positive,  $e>0$ (contradiction since $r\geq 2$).

It remains to consider the second case $t\geq 2,s=u=1$. Here
$$
\left[
  \begin{smallmatrix}
    r+2 & t & 0 & 1 \\
    r & t & t & 3 \\
    0 & 1 & t+2 & 1 \\
    r & t+2 & 1 & 1 \\
  \end{smallmatrix}
\right]
\left[
  \begin{smallmatrix}
    a \\
    b \\
    c \\
    d \\
  \end{smallmatrix}
\right]
=\nu\left[
      \begin{smallmatrix}
        1 \\
        1 \\
        1 \\
        1 \\
      \end{smallmatrix}
    \right],
$$
$$
\left[
  \begin{smallmatrix}
    a \\
    b \\
    c \\
    d \\
  \end{smallmatrix}
\right]
=\frac{\nu}{\det(A)}
\left[
  \begin{smallmatrix}
    3(1+t) \\
    2t-r+2 \\
    2(r+t+1) \\
    t(2-t)+(3-rt) \\
  \end{smallmatrix}
\right],
$$
and $\det(A)>0$ for all $r,t$. If we suppose that the graph is positive, then $d>0$, which is a contradiction as in fact, $d=\frac{\nu}{\det(A)}[t(2-t)+(3-rt)]<0$ ($t\geq 2$).
\end{proof}

\subsection{Five coherent components}\label{losejemplosde5comp}

\begin{lemma}\label{grafoscandidato} Under any of the following conditions,
\begin{enumerate}
  \item[(i)] $r,u\geq 2$
  \item[(ii)] $u=2,v\geq 15$
    \item[(iii)] $u=1$ and $r\geq 2$ \'o $v\geq 4$,
\end{enumerate}
the graph on the right in Figure \ref{Fig.} is not positive.
\end{lemma}

\begin{proof} Let us call $S_1,S_2,S_3,S_4,S_5$ the coherent components with $r,2,1,u,v$ vertices, respectively. According to Proposition \ref{ladossimilares}, there are only nine possibly different edge weights in this case: $a,b,c,d,e,f$ for the edges joining $S_1$ with $S_2$, $S_2$ with $S_3$, $S_3$ with $S_4$, $S_4$ with $S_5$, $S_2$ with $S_5$, and $S_2$ with $S_4$ respectively; and $g,h,i$ for the edges inside $S_2$, $S_4$ and $S_5$ respectively.

Suppose that in each case the graph is positive.

Let $r,u\geq 2$. Let us rewrite the system (\ref{equivalenciasistema}):
$$
\left[
  \begin{smallmatrix}
    0 & 0 & 0 & 2u & 4 & 0 & 0 & 0 & 2v-1 \\
    2r & 2 & 0 & 0 & 2v & 2u & 3 & 0 & 0 \\
    0 & 0 & 2 & 2v & 0 & 4 & 0 & 2u-1 & 0 \\
    3+r & 1 & 0 & 0 & v & u & 1 & 0 & 0 \\
    r & 4 & u & 0 & v & u & 1 & 0 & 0 \\
    0 & 2 & 2+u & v & 0 & 2 & 0 & u-1 & 0 \\
    r & 1 & 0 & u & 3+v & u & 1 & 0 & v-1 \\
    0 & 0 & 1 & u+v+1 & 2 & 2 & 0 & u-1 & v-1 \\
    r & 1 & 1 & v & v & 3+u & 1 & u-1 & 0 \\
  \end{smallmatrix}
\right]\left[
         \begin{smallmatrix}
           a \\
           b \\
           c \\
           d \\
           e \\
           f \\
           g \\
           h \\
           i \\
         \end{smallmatrix}
       \right]
       =\nu\left[
             \begin{smallmatrix}
               1 \\
               1 \\
               1 \\
               1 \\
               1 \\
               1 \\
               1 \\
               1 \\
               1 \\
             \end{smallmatrix}
           \right].
$$
Since the graph is positive, in particular, $g>0$. Straightforward calculations show that  $g=\frac{\nu(u+v+2)}{\det(A)}[u^2(2r-3)+uv(2r-3)+u(5r-9)+3v(2r-1)+(3r-6)]$, which is negative because $\det(A)<0$ for all $r,u,v$, and $u^2(2r-3)+uv(2r-3)+u(5r-9)+3v(2r-1)+(3r-6)>0$ since $r\geq 2$. In this case, the graph is not positive.

For the second case, $u=2$ and $v\geq 15$, we use the same matrix entry $g$ already calculated. If $u=2$, $g=\frac{\nu(v+4)}{\det(A)}[(10rv-9v)+(21r-36)]$. Then, $g\leq \frac{\nu(v+4)}{\det(A)}[(10v-9v)+(21-36)]\leq 0$ since $v\geq 15$. Therefore, the graph is not positive.

Finally, let $u=1$. In this case
$$
\left[
  \begin{smallmatrix}
    0 & 0 & 0 & 2 & 4 & 0 & 0 & 2v-1 \\
    0 & 2 & 3 & v & 0 & 2 & 0 & 0 \\
    2r & 2 & 0 & 0 & 2v & 2 & 3 & 0 \\
    3+r & 1 & 0 & 0 & v & 1 & 1 & 0 \\
    r & 4 & 1 & 0 & v & 1 & 1 & 0 \\
    0 & 0 & 1 & 2+v & 2 & 2 & 0 & v-1 \\
    r & 1 & 1 & v & v & 4 & 1 & 0 \\
    r & 1 & 0 & 1 & 3+v & 1 & 1 & v-1 \\
  \end{smallmatrix}
\right]\left[
         \begin{smallmatrix}
           a \\
           b \\
           c \\
           d \\
           e \\
           f \\
           g \\
           i \\
         \end{smallmatrix}
       \right]
       =\nu\left[
             \begin{smallmatrix}
               1 \\
               1 \\
               1 \\
               1 \\
               1 \\
               1 \\
               1 \\
               1 \\
             \end{smallmatrix}
           \right].
$$
Since the graph is positive, then $g$ is positive. Here $g=\frac{2\nu(3+v)}{\det(A)}[(5r-9)+v(4r-3)]=\frac{2\nu(3+v)}{\det(A)}[(4rv-3v)+(5r-9)]$, where $\det(A)$ is negative for all $r,v$. If $r\geq 2$, $g<0$ since $(5r-9)+v(4r-3)>0$; and if $v\geq 4$, $g\leq\frac{2\nu(3+v)}{\det(A)}[(4v-3v)+(5-9)]\leq 0$.

In all cases we arrive at a contradiction.
\end{proof}

\section{Proof of Theorem \ref{thm}}

Let $q\geq 21$ be fixed, and consider the following coherence graph,
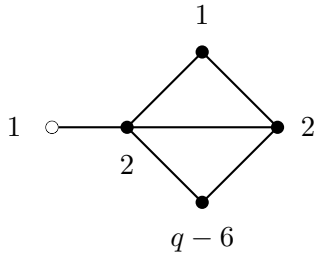
\begin{center}
\begin{pspicture}(0,-2)(3,2)
\psline(0,0)(1,0)(2,1)(3,0)(1,0)(2,-1)(3,0)
\psdots[dotstyle=o,dotsize=5pt](0,0)
\psdots[dotstyle=*,dotsize=5pt](1,0)
\psdots[dotstyle=*,dotsize=5pt](2,1)
\psdots[dotstyle=*,dotsize=5pt](3,0)
\psdots[dotstyle=*,dotsize=5pt](2,-1)
\rput[d](-.5,0){$1$}
\rput[d](1,-.5){$2$}
\rput[c](2,1.5){$1$}
\rput[r](3.5,0){$2$}
\rput[d](2,-1.5){$q-6$}
\end{pspicture}
\end{center}
thus the corresponding graph must be of the form
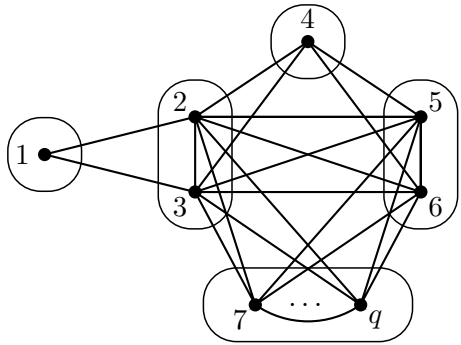
\begin{center}
\begin{pspicture}(0,-.5)(5,5)
\psline(2,2)(2,3)(5,3)(5,2)(2,2)(5,3)(5,2)
\psline(2,3)(3.5,4)(5,3)
\psline(2,2)(0,2.5)(2,3)
\psline(5,2)(2,3)(5,3)
\psline(2,2)(3.5,4)(5,2)
\psline(2,2)(2.8,.5)(2,3)(4.2,.5)(2,2)
\psline(5,2)(4.2,.5)(5,3)(2.8,.5)(5,2)
\psdots[dotstyle=*,dotsize=5pt](2,2)
\psdots[dotstyle=*,dotsize=5pt](2,3)
\psdots[dotstyle=*,dotsize=5pt](5,2)
\psdots[dotstyle=*,dotsize=5pt](5,3)
\psdots[dotstyle=*,dotsize=5pt](0,2.5)
\psdots[dotstyle=*,dotsize=5pt](3.5,4)
\psdots[dotstyle=*,dotsize=5pt](2.8,.5)
\psdots[dotstyle=*,dotsize=5pt](4.2,.5)

\rput[r](-.2,2.5){$1$}
\rput[c](1.8,1.8){$3$}
\rput[c](1.8,3.2){$2$}
\rput[c](5.2,1.8){$6$}
\rput[c](5.2,3.2){$5$}
\rput[c](3.5,4.3){$4$}
\rput[c](2.6,.3){$7$}

\psframe[linewidth=.2mm,framearc=.9,
fillcolor=lightgray](1.5,1.5)(2.5,3.5)
\psframe[linewidth=.2mm,framearc=.9,
fillcolor=lightgray](4.5,1.5,)(5.5,3.5)
\psframe[linewidth=.2mm,framearc=.9,
fillcolor=lightgray](3,3.5)(4,4.5)
\psframe[linewidth=.2mm,framearc=.9,
fillcolor=lightgray](-.5,2)(.5,3)
\psframe[linewidth=.2mm,framearc=.9,
fillcolor=lightgray](2.1,0)(4.9,1)

\psarc(3.5,1.5){1.220655562}{235}{305}

\rput[c](4.4,.3){$q$}
\rput[c](3.5,.5){$\dots$}
\end{pspicture}
\end{center}
$\GRA$ is a graph as on the right in Figure \ref{Fig.} with $r=1$, $s=2$, $t=1$, $u=2$ and $v=q-6$. According to Lemma \ref{grafoscandidato}, part (ii), $\GRA$ is not positive because $u=2$ and $v=q-6\geq 15$. By Proposition \ref{criteriopositividad}, since $\GRA$ is not positive, the Lie algebra associated $\ngo_{\GRA}$ is not an Einstein nilradical.

Let us consider the following ordered set of certain edges of the graph $\GRA$:
\begin{equation*}
\begin{split}
H=&\{\{4,5\},\{4,6\};\\
&\{5,6\},\dots,\{5,q\};\{6,7\},\dots,\{6,q\};\dots;\{q-2,q-1\},\{q-2,q\};\{q-1,q\};\\
&\{2,5\},\dots,\{2,q\};\\
&\{2,3\},\{1,2\}\};
\end{split}
\end{equation*}
Let us now define a sequence of 2-step nilpotent Lie algebras by
$$
\left\{\begin{array}{ccl}
  \ngo_0 & := & \ngo_{\GRA}\,; \\
  \ngo_l & := & \ngo_{\GRA(l)}\,,\qquad 1\leq l \leq |H|;
\end{array}
\right.
$$
where $\GRA(l)$ is the graph which we obtain by deleting the first $l$ edges in $H$. If the edge at the place $l$ in $H$ is $\{i,j\}$, we define $\GRA(i,j):=\GRA(l)$ and $\ngo(i,j):=\ngo_l$.

We shall prove that $\ngo_l$ is not an Einstein nilradical for all $l=1,\dots,|H|$.

By deleting the edge $\{4,5\}$ in $\GRA$, we obtain $\GRA(1)$, which is a graph as on the right in Figure
\ref{Fig.} with $r=u=1,v=q-5$. According to Lemma \ref{grafoscandidato}, part (iii), the Lie algebra associated $\ngo_1$ to $\GRA(1)$ is not an Einstein nilradical because $u=1$ and $v=q-5\geq 16\geq 4$.

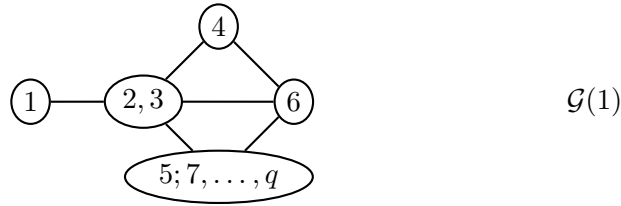
\begin{center}
\begin{pspicture}(0,-2)(4,2)
\rput(-.5,0){\ovalnode{a}{$1$}}
\rput(1,0){\ovalnode{b}{$2,3$}}
\rput(2,1){\ovalnode{c}{$4$}}
\rput(3,0){\ovalnode{d}{$6$}}
\rput(2,-1){\ovalnode{e}{$5;7,\dots,q$}}

\ncline{-}{a}{b}
\ncline{-}{b}{c}
\ncline{-}{c}{d}
\ncline{-}{b}{d}
\ncline{-}{b}{e}
\ncline{-}{d}{e}

\rput[c](7,0){$\GRA(1)$}
\end{pspicture}
\end{center}

Now by deleting the edge $\{4,6\}$ in $\GRA(1)$ we obtain $\GRA(2)$ which is a graph with 3 coherent components with $r=s=2,t=q-4$, as we have considered in \ref{o-*-*}. Then, $\GRA(2)$ is not positive because the positivity condition $(s+t)(s-r)>(r-1)(t-1)$ does not hold: $(s+t)0=(s+t)(s-r)\leq (r-1)(t-1)=q-5$. Consequently, by the Proposition \ref{criteriopositividad}, $\ngo_2$ is not an Einstein nilradical.
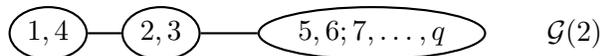
\begin{center}
\begin{pspicture}(0,-1)(4,1)
\rput(0,0){\ovalnode{a}{$1,4$}}
\rput(1.5,0){\ovalnode{b}{$2,3$}}
\rput(4.3,0){\ovalnode{e}{$5,6;7,\dots,q$}}
\ncline{-}{a}{b}
\ncline{-}{b}{e}

\rput[c](7,0){$\GRA(2)$}
\end{pspicture}
\end{center}

By now we have already obtained graphs $\GRA(5,6),\GRA(5,7),\dots,\GRA(5,q-1),\GRA(5,q)$.

If we delete the edge $\{5,6\}$ in $\GRA(2)$, we obtain $\GRA(5,6)$, which is a graph with 4 coherent components as in the middle in Figure \ref{Fig.} with $r=2$ ,$s=t=2$ and $u=q-6$. By Lemma \ref{4comp 2}, part (i), $\GRA(5,6)$ is not positive because $s=t=2$ and $r\geq 2$. Then, according to Proposition \ref{criteriopositividad}, $\ngo(5,6)$ follows that it is not an Einstein nilradical.
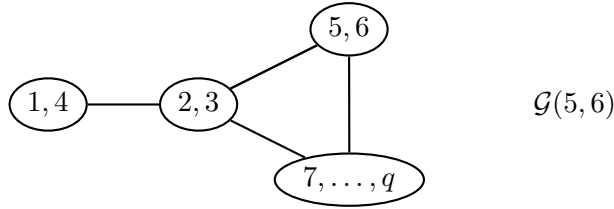
\begin{center}
\begin{pspicture}(0,-2)(4,2)
\rput(0,0){\ovalnode{a}{$1,4$}}
\rput(2,0){\ovalnode{b}{$2,3$}}
\rput(4,1){\ovalnode{c}{$5,6$}}
\rput(4,-1){\ovalnode{d}{$7,\dots,q$}}
\ncline{-}{a}{b}
\ncline{-}{b}{c}
\ncline{-}{c}{d}
\ncline{-}{b}{c}
\ncline{-}{b}{d}
\rput[c](7,0){$\GRA(5,6)$}
\end{pspicture}
\end{center}

If we successively delete the edges $(5,7),\dots,(5,j)$ in $\GRA(5,6)$, with $7\leq j\leq q-1$, we obtain  $\GRA(5,j)$, which is a graph with 5 coherent components as on the right in Figure \ref{Fig.} with $r=2$, $u=q-j$ and $v=j-5$. In order to prove that these graphs are not positive, we apply Lemma \ref{grafoscandidato}. We must distinguish two cases: $7\leq j\leq q-2$ and $j=q-1$.

If $j\leq q-2$, then, according to part (i) in Lemma \ref{grafoscandidato}, $\GRA(5,j)$ is not positive because $r\geq 2$ and $u=q-j\geq q-(q-2)=2$. If $j=q-1$, according to part (iii) in Lemma \ref{grafoscandidato}, $\GRA(5,q-1)$ is not positive because $u=q-(q-1)=1$ and $r\geq 2$.

Therefore, graphs $\GRA(5,7),\dots,\GRA(5,q-1)$ are not positive, and then the Lie algebras associated $\ngo(5,7),\dots,\ngo(5,q-1)$ are not Einstein nilradicals, according to Proposition \ref{criteriopositividad}.
\begin{center}
\begin{pspicture}(0,-2)(4,2)
\rput(-.5,0){\ovalnode{a}{$1,4$}}
\rput(1,0){\ovalnode{b}{$2,3$}}
\rput(2.5,1){\ovalnode{c}{$5$}}
\rput(4,0){\ovalnode{d}{$j+1,\dots,q$}}
\rput(2.5,-1){\ovalnode{e}{$6,\dots,j$}}

\ncline{-}{a}{b}
\ncline{-}{b}{c}
\ncline{-}{c}{d}
\ncline{-}{b}{d}
\ncline{-}{b}{e}
\ncline{-}{d}{e}

\rput[c](7,0){$\GRA(5,j)$}
\rput[c](7.5,-.5){$7\leq j\leq q-1$}
\end{pspicture}
\end{center}

By deleting the edge $\{5,q\}$ in $\GRA(5,q-1)$ we obtain $\GRA(5,q)$ which is a graph with 3 coherent components
with $r=3,s=2,t=q-5$ (see \ref{o-*-*}), and it follows that it is a non-positive graph because $(q-3)(-1)=(s+t)(s-r)\leq (r-1)(t-1)=2(q-6)$. Then, $\ngo(5,q)$ is not an Einstein nilradical (Proposition\ref{criteriopositividad}).
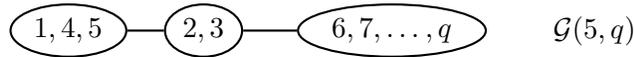
\begin{center}
\begin{pspicture}(0,-1)(4,1)
\rput(0,0){\ovalnode{a}{$1,4,5$}}
\rput(1.8,0){\ovalnode{b}{$2,3$}}
\rput(4.3,0){\ovalnode{e}{$6,7,\dots,q$}}
\ncline{-}{a}{b}
\ncline{-}{b}{e}

\rput[c](7,0){$\GRA(5,q)$}
\end{pspicture}
\end{center}

Graphs $\GRA(6,7),\GRA(6,8),\dots,\GRA(6,q-1),\GRA(6,q)$ are the following:
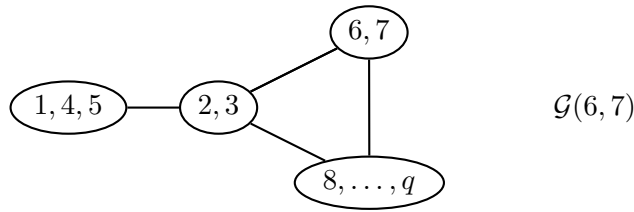
\begin{center}
\begin{pspicture}(0,-2)(4,2)
\rput(0,0){\ovalnode{a}{$1,4,5$}}
\rput(2,0){\ovalnode{b}{$2,3$}}
\rput(4,1){\ovalnode{c}{$6,7$}}
\rput(4,-1){\ovalnode{d}{$8,\dots,q$}}
\ncline{-}{a}{b}
\ncline{-}{b}{c}
\ncline{-}{c}{d}
\ncline{-}{b}{c}
\ncline{-}{b}{d}
\rput[c](7,0){$\GRA(6,7)$}
\end{pspicture}
\end{center}
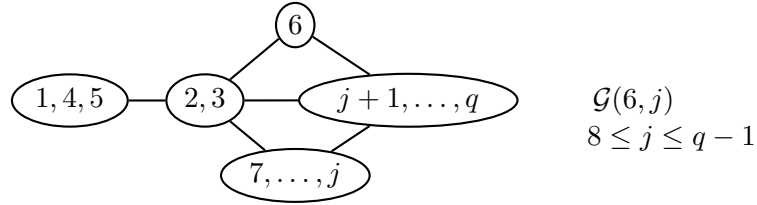
\begin{center}
\begin{pspicture}(0,-2)(4,2)
\rput(-.5,0){\ovalnode{a}{$1,4,5$}}
\rput(1.3,0){\ovalnode{b}{$2,3$}}
\rput(2.5,1){\ovalnode{c}{$6$}}
\rput(4,0){\ovalnode{d}{$j+1,\dots,q$}}
\rput(2.5,-1){\ovalnode{e}{$7,\dots,j$}}

\ncline{-}{a}{b}
\ncline{-}{b}{c}
\ncline{-}{c}{d}
\ncline{-}{b}{d}
\ncline{-}{b}{e}
\ncline{-}{d}{e}

\rput[c](7,0){$\GRA(6,j)$}
\rput[c](7.5,-.5){$8\leq j\leq q-1$}
\end{pspicture}
\end{center}
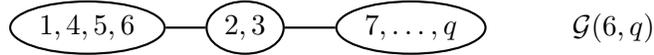
\begin{center}
\begin{pspicture}(0,-1)(4,1)
\rput(0,0){\ovalnode{a}{$1,4,5,6$}}
\rput(2.1,0){\ovalnode{b}{$2,3$}}
\rput(4.3,0){\ovalnode{e}{$7,\dots,q$}}
\ncline{-}{a}{b}
\ncline{-}{b}{e}

\rput[c](7,0){$\GRA(6,q)$}
\end{pspicture}
\end{center}

Note that the structure of graphs $\GRA(5,6),\dots,\GRA(5,q)$ is the same as in $\GRA(6,7),\dots,\GRA(6,q)$; and $\GRA(7,8),\dots,\GRA(7,q)$; and so on until $\GRA(q-3,q-2),\GRA(q-3,q-1),\GRA(q-3,q)$.

In general, in graphs $\GRA(i,i+1),\GRA(i,i+2),\dots,\GRA(i,q-1),\GRA(i,q)$, with $5\leq i\leq q-3$, we must distinguish three cases: the first graph, the last graph, and the ones in the middle:
\begin{itemize}
  \item $\GRA(i,i+1)$ is a graph with 4 coherent components as in the middle in Figure \ref{Fig.} with $r=i-2$, $s=2$, $t=2$ and $u=q-i-1$; and thus it is not positive according to Lemma \ref{4comp 2}, part (i) ($s=t=2$ and $r=i-2\geq 3\geq 2$). Then, $\ngo(i,i+1)$ is not an Einstein nilradical by the Proposition \ref{criteriopositividad}.
      \begin{center}
\begin{pspicture}(0,-2)(4,2)
\rput(-1,0){\ovalnode{a}{$1;4,\dots,i-1$}}
\rput(2,0){\ovalnode{b}{$2,3$}}
\rput(4,1){\ovalnode{c}{$i,i+1$}}
\rput(4,-1){\ovalnode{d}{$i+2,\dots,q$}}
\ncline{-}{a}{b}
\ncline{-}{b}{c}
\ncline{-}{c}{d}
\ncline{-}{b}{c}
\ncline{-}{b}{d}
\rput[c](7,0){$\GRA(i,j)$}
\rput[c](7,-.5){$5\leq i\leq q-3$}
\rput[c](7,-1){$j=i+1$}

\end{pspicture}
\end{center}
  \item $\GRA(i,j)$ with $i+2\leq j\leq q-1$ are graphs with 5 coherent components as on the right in Figure \ref{Fig.} with $r=i-3$, $u=q-j$ and $v=j-i$, and by Lemma \ref{grafoscandidato} all these graphs are not positive; for, if $j\leq q-2$, we apply the part (i) in the Lemma \ref{grafoscandidato} since $r=i-3\geq 2$ and $u=q-j\geq q-(q-2)=2$; and if $j=q-1$, we apply part (iii) in Lemma \ref{grafoscandidato} because $u=1$ and $r=i-3\geq 2$. Then, according to Proposition \ref{criteriopositividad}, $\ngo(i,j)$ is not an Einstein nilradical, for all $i=7,\dots,q-1$.
      
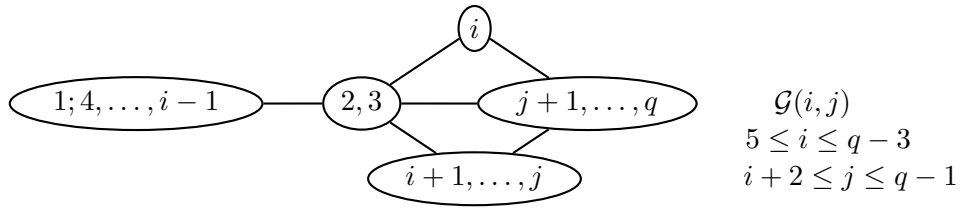
\begin{center}
\begin{pspicture}(0,-2)(4,2)
\rput(-2,0){\ovalnode{a}{$1;4,\dots,i-1$}}
\rput(1,0){\ovalnode{b}{$2,3$}}
\rput(2.5,1){\ovalnode{c}{$i$}}
\rput(4,0){\ovalnode{d}{$j+1,\dots,q$}}
\rput(2.5,-1){\ovalnode{e}{$i+1,\dots,j$}}

\ncline{-}{a}{b}
\ncline{-}{b}{c}
\ncline{-}{c}{d}
\ncline{-}{b}{d}
\ncline{-}{b}{e}
\ncline{-}{d}{e}

\rput[c](7,0){$\GRA(i,j)$}
\rput[c](7.2,-.5){$5\leq i\leq q-3$}
\rput[c](7.5,-1){$i+2\leq j\leq q-1$}
\end{pspicture}
\end{center}
  \item $\GRA(i,q)$ is a graph as in \ref{o-*-*} with $r=i-2$, $s=2$ and $t=q-i$, which is not positive because the positivity condition $(s+t)(s-r)>(r-1)(t-1)$ does not hold: $(s+t)(s-r)=(q-i+2)(4-i)\leq 0\leq (i-3)(q-i-1)=(r-1)(t-1)$.
  
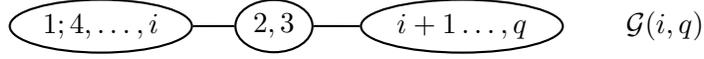
\begin{center}
\begin{pspicture}(0,-1)(4,1)
\rput(-.5,0){\ovalnode{a}{$1;4,\dots,i$}}
\rput(1.8,0){\ovalnode{b}{$2,3$}}
\rput(4.3,0){\ovalnode{e}{$i+1\dots,q$}}
\ncline{-}{a}{b}
\ncline{-}{b}{e}

\rput[c](7,0){$\GRA(i,q)$}
\end{pspicture}
\end{center}
\end{itemize}

Graphs $\GRA(q-2,q-1),\GRA(q-2,q)$ and $\GRA(q-1,q)$ are:
      
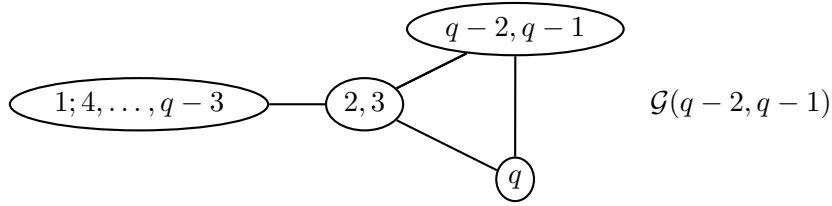
\begin{center}
\begin{pspicture}(0,-2)(4,2)
\rput(-1,0){\ovalnode{a}{$1;4,\dots,q-3$}}
\rput(2,0){\ovalnode{b}{$2,3$}}
\rput(4,1){\ovalnode{c}{$q-2,q-1$}}
\rput(4,-1){\ovalnode{d}{$q$}}
\ncline{-}{a}{b}
\ncline{-}{b}{c}
\ncline{-}{c}{d}
\ncline{-}{b}{c}
\ncline{-}{b}{d}
\rput[c](7,0){$\GRA(q-2,q-1)$}
\end{pspicture}
\end{center}
 
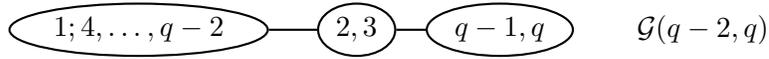
\begin{center}
\begin{pspicture}(0,-1)(4,1)
\rput(-.5,0){\ovalnode{a}{$1;4,\dots,q-2$}}
\rput(2.4,0){\ovalnode{b}{$2,3$}}
\rput(4.3,0){\ovalnode{e}{$q-1,q$}}
\ncline{-}{a}{b}
\ncline{-}{b}{e}

\rput[c](7,0){$\GRA(q-2,q)$}
\end{pspicture}
\end{center}
  
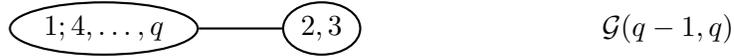
\begin{center}
\begin{pspicture}(0,-1)(4,1)
\rput(-.5,0){\ovalnode{a}{$1;4,\dots,q$}}
\rput(2.4,0){\ovalnode{b}{$2,3$}}
\ncline{-}{a}{b}

\rput[c](7,0){$\GRA(q-1,q)$}
\end{pspicture}
\end{center}

$\GRA(q-2,q-1)$ is a graph with 4 coherent components as in the middle in Figure \ref{Fig.} with $r=q-5$, $s=2$, $t=2$ and $u=q-5$, and thus it is not positive by part (i) in the Lemma \ref{4comp 2} because $s=t=2$ and $r=q-5\geq 16\geq 2$; and $\GRA(q-2,q)$ is a graph as in \ref{o-*-*} with $r=q-4$, $s=2$ and $t=2$ which is not positive because $(s+t)(s-r)=4(2-q)\leq 0\leq q-5=(r-1)(t-1)$. Then, the Lie algebras associated, $\ngo(q-2,q-1)$ and $\ngo(q-2,q)$ are not Einstein nilradicals (Proposition \ref{criteriopositividad}). Similarly, the Lie algebra $\ngo(q-1,q)$ associated with graph $\GRA(q-1,q)$ is not an Einstein nilradical because $\GRA(q-1,q)$ is not positive (see \ref{o-*}). $\star$

Now from $\GRA(q-1,q)$ let us obtain graphs $\GRA(2,5),\dots,\GRA(2,q)$. By successively deleting the edges $\{2,5\},\dots,\{2,j\}$, we have $\GRA(2,j)$, $5\leq j\leq q$. Each $\GRA(2,j)$ is a graph with 4 coherent components as in middle in Figure \ref{Fig.} with $r=j-4$, $s=1$, $t=q-j+2$ and $u=1$. (Convention: in the figure below $\{1,4,i+1,\dots,q\}=\{1,4\}$ if $j=q$.) Then, by Lemma \ref{4comp 2}, part (ii), $\GRA(2,j)$ is non-positive for all $j=5,\dots,q$  because $s=u=1$ and $t=q-j+2\geq q-q+2=2$. Then, according to Proposition \ref{criteriopositividad}, $\ngo(2,j)$ is not an Einstein nilradical for all $j=5,\dots,q$.
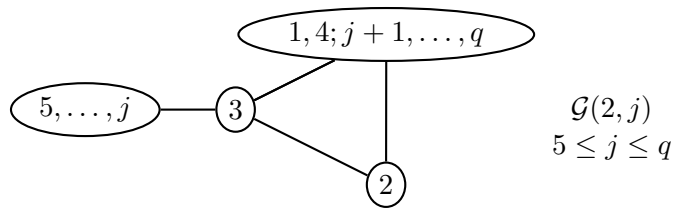
      \begin{center}
\begin{pspicture}(0,-2)(4,2)
\rput(0,0){\ovalnode{a}{$5,\dots,j$}}
\rput(2,0){\ovalnode{b}{$3$}}
\rput(4,1){\ovalnode{c}{$1,4;j+1,\dots,q$}}
\rput(4,-1){\ovalnode{d}{$2$}}
\ncline{-}{a}{b}
\ncline{-}{b}{c}
\ncline{-}{c}{d}
\ncline{-}{b}{c}
\ncline{-}{b}{d}
\rput[c](7,0){$\GRA(2,j)$}
\rput[c](7,-.5){$5\leq j\leq q$}
\end{pspicture}
\end{center}

If we delete the edge $\{2,3\}$ in $\GRA(2,q)$ we have:
      
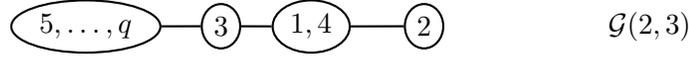
\begin{center}
\begin{pspicture}(0,-2)(4,2)
\rput(-.5,0){\ovalnode{a}{$5,\dots,q$}}
\rput(1.3,0){\ovalnode{b}{$3$}}
\rput(2.5,0){\ovalnode{c}{$1,4$}}
\rput(4,0){\ovalnode{d}{$2$}}
\ncline{-}{a}{b}
\ncline{-}{b}{c}
\ncline{-}{c}{d}
\rput[c](7,0){$\GRA(2,3)$}
\end{pspicture}
\end{center}
By deleting the edge $\{1,2\}$ in $\GRA(2,3)$ we have:
      
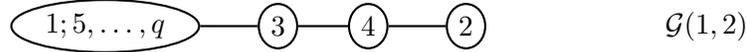
\begin{center}
\begin{pspicture}(0,-2)(4,2)
\rput(-1,0){\ovalnode{a}{$1;5,\dots,q$}}
\rput(1.3,0){\ovalnode{b}{$3$}}
\rput(2.5,0){\ovalnode{c}{$4$}}
\rput(3.8,0){\ovalnode{d}{$2$}}
\ncline{-}{a}{b}
\ncline{-}{b}{c}
\ncline{-}{c}{d}
\rput[c](7,0){$\GRA(1,2)$}
\end{pspicture}
\end{center}
Graphs $\GRA(2,3)$ and $\GRA(1,2)$ are as on the left in Figure \ref{Fig.} with $s=1,2$ and $u=q-4$. Then by Lemma \ref{4comp 1}, $\GRA(2,3)$ and $\GRA(1,2)$ are not positive graphs because $u=q-4\geq 6$. Therefore, the Lie algebras associated, $\ngo(2,3)$ and $\ngo(1,2)$, are not Einstein nilradicals (Proposition \ref{criteriopositividad}).

Remember that if a graph is connected, then  the Lie algebra associated is {\it indecomposable}. It can be proved that $\GRA(1,2)=\GRA(|H|)$ is connected. Then all the graphs $\GRA(l)$, $l=1,\dots,|H|$, are also connected (because we obtain them by adding edges in $\GRA(1,2)$). Therefore, the Lie algebra $\ngo_l$ is indecomposable for all $l=0,1,\dots,|H|$.

In addition, remember that if a connected graph $\GRA$ has $p$ edges and $q$ vertices, its 2-step nilpotent Lie algebra associated is of type $(p,q)$. The starting point was graph $\GRA$, which has $p=D_q-2q+9$ edges and $q$ vertices, and so $\ngo_0=\ngo_{\GRA}$ is a Lie algebra of type $(D_q-2q+9,\,q)$. By successively deleting all the elements of $H$ we obtain the last graph $\GRA(1,2)$, which is a graph with $q$ vertices and $p=q-1$ edges. Then $\ngo(1,2)=\ngo_{\GRA(1,2)}$ is a Lie algebra of type $(q-1,q)$. Obviously each intermediate graph between $\GRA$ and $\GRA(1,2)$ has $q$ (fixed) vertices and $p$ edges, where $q-1\leq p \leq D_q-2q+9$. Therefore we have proved that for all $(p,q)$ satisfying $21\leq q$ and $q-1\leq p\leq D_q-2q+9$, choosing $l=D_q-2q+9-p$, $\ngo_l$ is a 2-step nilpotent Lie algebra of type $(p,q)$, that is indecomposable and in addition it is not an Einstein nilradical. This concludes the proof of Theorem \ref{thm}.


\begin{thebibliography}{XXXX}

\bibitem[C+]{libro}  {\sc B. Chow, S.-C. Chu, D. Glickenstein, C.
Guenther, J. Isenberg, T, Ivey, D. Knopf, P. Lu, F. Luo, L. Ni}, The Ricci flow:
Techniques and Applications, Part I: Geometric Aspects, {\it AMS Math. Surv. Mon.}
{\bf 135} (2007), Amer. Math. Soc., Providence.

\bibitem[DM]{DnMnk} {\sc S. G. Dani, M. Mainkar,}
 Anosov automorphisms on compact nilmanifolds associated with
graphs, {\it Trans. Amer. Math. Soc.} 357 (2004), 2235-2251.

\bibitem[F]{Frn} {\sc E. Fernandez Culma}, Classification of $7$-dimensional Einstein Nilradicals, preprint 2011 (arXiv).

\bibitem[J1]{Jbl} {\sc M. Jablonski}, Moduli of Einstein and non-Einstein nilradicals.
{\it Geom. Dedicata} 152 (2011), 63--84.

\bibitem[J2]{Jbl2} \bysame, Detecting orbits along subvarieties via the moment map,
 {\it M\"{u}nster J. Math.}, 3 (2010), 67--88.

\bibitem[La]{Lfn} {\sc R. Lafuente}, Solvsolitons associated with graphs, {\it Adv. Geom.} (2011), in press.

\bibitem[L1]{Lrt} {\sc J. Lauret}, Ricci soliton homogeneous nilmanifolds,
{\it Math. Ann.} \textbf{319} (2001), 715-733.

\bibitem[L2]{cruzchica}  \bysame, Einstein solvmanifolds and nilsolitons, {\it Contemp. Math.} {\bf 491} (2009), 1-35.

\bibitem[LW]{LrtWll} {\sc J. Lauret, C. E. Will}, Einstein solvmanifolds: existence
and non-existence questions, {\it Math. Annalen}, 350 (2011), 199-225.

\bibitem[N1]{Nkl1} {\sc Y. Nikolayevsky}, Einstein solvmanifolds and the
pre-Einstein derivation. {\it Trans. Amer. Math. Soc.}, 363 (2011), 3935-395.

\bibitem[N2]{Nkl2} \bysame, Einstein solvmanifolds attached to two-step nilradicals.
{\it Math. Zeit.}, in press.

\bibitem[P]{Pyn} {\sc T. Payne}, The topology of the set of nonsoliton Lie algebras in the moduli space of nilpotent Lie algebras, preprint 2011 (arXiv).

\bibitem[W]{Wll} {\sc C. E. Will}, Rank-one Einstein solvmanifolds of dimension 7, {\it Diff. Geom.
Appl.}, 19 (2003), 307-318.






\end{thebibliography}
\end{document}